\documentclass[12pt]{amsart}
\usepackage{hyperref}
\usepackage{amsmath, amsthm,amstext}
\usepackage{amssymb, array, amsfonts}
\usepackage[english]{babel}
\usepackage{url}
\usepackage{amsrefs}
\usepackage{float}
\usepackage{graphics}
\usepackage{graphicx}
   \usepackage{enumerate}
\numberwithin{equation}{section}

\usepackage[margin=1.2in]{geometry}

\def \ep{\varepsilon}

\renewcommand{\l}{\left}
\renewcommand{\r}{\right}

\def \C{\mathbb{C}}

\def \Q{\mathbb{Q}}
\def \rL{\mathrm L}

\def \N{\mathbb{N}}
\def \M2{\mathrm{M}_2}
\def \R{\mathbb{R}}
\def \Z{\mathbb{Z}}

\def \cV{\mathcal{V}}

\def \sl2r{\mathrm{SL}(2,\R)}

\newcommand{\beq}{\begin{equation}}
\newcommand{\eeq}{\end{equation}}

\def\Ell{\mathrm{L}}

\def\tr{\operatorname{tr}}

\newcommand{\eqdef}{\stackrel{\rm def}{=\kern-3.6pt=}}

\theoremstyle{plain}
\newtheorem{theorem}{\bf Theorem}[section]
\newtheorem{lemma}[theorem]{\bf Lemma}

\newtheorem{prop}[theorem]{\bf Proposition}
\newtheorem{cor}[theorem]{\bf Corollary}

\theoremstyle{definition}

\theoremstyle{remark}
\newtheorem{remark}[theorem]{\bf Remark}

\renewcommand{\le}{\leqslant}
\renewcommand{\ge}{\geqslant}

\newcommand{\dist}{\mathop{\mathrm{dist}}\nolimits}
\newcommand{\var}{\mathop{\mathrm{Var}}\nolimits}

\renewcommand{\qed}{\vrule height7pt width5pt depth0pt}
\title
[Sharp arithmetic delocalization]{Sharp arithmetic delocalization for quasiperiodic operators with potentials of semi-bounded variation.}
\author[S. Jitomirskaya]{Svetlana Jitomirskaya}
\address{Department of Mathematics, University of California, Berkeley, CA 94720, USA.}
\email{sjitomi@berkeley.edu}

\author[I. Kachkovskiy]{Ilya Kachkovskiy}
\address{Department of Mathematics,
	Michigan State University,
	Wells Hall, 619 Red Cedar Road,
	East Lansing, MI 48824,
	United States of America}
\email{ikachkov@msu.edu}
\date{}

\begin{document}
\begin{abstract}
We obtain the sharp arithmetic Gordon's theorem: that is, absence of
eigenvalues on the
set of energies with Lyapunov exponent bounded by the exponential rate
of approximation of frequency by the rationals, for a large class of one-dimensional
quasiperiodic Schr\"odinger operators,  with
no (modulus of) continuity required. The class includes all unbounded
monotone potentials with finite Lyapunov exponents and all potentials
of bounded variation. The main tool is a new uniform upper bound on
iterates of cocycles of bounded variation.
\end{abstract}
\maketitle
\section{Introduction and the main result}
Gordon’s trick \cite{Gor} has been used to prove absence of point
spectra for one-dimensional differential and finite difference operators since \cite{as} (see also
\cite{Damanik}). Roughly, it requires infinitely many 
pinned triple almost repetitions of growing blocks of the potential. It has
been particularly useful for the study of quasiperiodic operators
$H(x)$  on $\ell^2(\Z)$  given by
\beq
\label{eq_h_def}
(H(x)\psi)(n)=\psi(n+1)+\psi(n-1)+f(x+n\alpha)\psi(n).
\eeq
For these operators, such almost repetitions are guaranteed, for all $x$, for frequencies $\alpha \in \R\setminus\Q$
that are sufficiently well approximated by the
rationals\footnote{Conditions of this type are usually called Liouville.} if $f$ is uniformly
continuous. Since for $\alpha\in \R\setminus\Q$ the spectrum of $H(x)$
is
absolutely continuous and absence of point spectrum is Baire generic
\cite{wonder}, for continuous $f$ it of course immediately follows that there is no point
spectrum for $\alpha$ satisfying  {\it some} Liouville-type
condition. The advantage of a Gordon-type argument is that it
quantifies the Liouvilleness.

Moreover, Gordon's trick has also been successfully applied
to operators \eqref{eq_h_def} with discontinuous $f,$ where the Baire
generic argument of \cite{wonder} does not directly apply. For discontinuous $f,$ other than
in the Sturmian case where there are additional considerations
\cite{sturmuniform}, the results usually hold only for a full Lebesgue measure set of phases $x$ rather than all
values of $x$, see, for example, \cite{simmar}. An ultimate result of this kind was
proved by Gordon \cite{Gor2} who showed absence of point spectrum
for generic $\alpha$ for {\it any measurable} $f.$ 
The result of \cite{Gor2}  also specifies the required Liouvilleness,
but in a rather implicit  way that, importantly, depends substantially
on $f,$ through the measure-theoretic rate of it's growth and its
measure-theoretic modulus of continuity.

At the same time, for Lipschitz $f$ already the original Gordon's
trick gives the Liouvilleness requirement fully in terms of the
Lyapunov exponent.

Let $f\colon \R\to [-\infty,+\infty)$ be a $1$-periodic measurable function, satisfying
\beq
\label{eq_finite_integral}
\int_0^1 \log(1+|f(x)|)\,dx<+\infty.
\eeq
Define the one-step and $n$-step transfer matrices associated to the Schr\"odinger eigenvalue equation $H(x)\psi=E\psi$ by
$$
S^{f,E}(x):=\begin{pmatrix}E-f(x)&-1\\1&0 \end{pmatrix},\quad S_n^{f,E}(x):=\prod_{k=n-1}^0 S^{f,E}(x+k\alpha).
$$
Let
\beq
\label{gamma_def1}
L(E):=L(\alpha,S^{f,E})=\lim_{n\to\infty}\frac{1}{n}\int_{[0,1)}\log\l\|S_{n}^{f,E}(x)\r\|\,dx.
\eeq
be the Lyapunov exponent of the corresponding Schr\"odinger cocycle $(\alpha,S^{f,E})$.
Due to \eqref{eq_finite_integral} (see also \eqref{l}), $L(E)$ is well defined and finite\footnote{One can also check that \eqref{eq_finite_integral} is necessary for $L(E)$ to be finite for all $E\in \R$.}.

Let $\frac{p_k}{q_k}\to \alpha$ be the sequence of the continued
fraction approximants. Define the measure of Liouvilleness\footnote{The quantity $\beta(\alpha)+1$ is sometimes referred to as the irrationality measure of $\alpha$.}
$\beta(\alpha)\in [0,\infty]$ by
\beq\label{beta}
\beta(\alpha):=\limsup_{k\to +\infty}\frac{\log q_{k+1}}{q_k}.
\eeq

For Lipschitz $f,$ the Gordon-based arguments provide a proof that $E$ cannot be an
eigenvalue of $H(x)$ entirely though the interplay
between $L(E)$ and $\beta(\alpha).$  Classical results
\cite{as} showed, effectively, that $E$ cannot be an eigenvalue if
$L(E)<C\beta(\alpha)$. It has been conjectured in
\cite{simmar,jcongr} and proved in \cite{JL,AYZ,JL2} that for the Maryland and almost
Mathieu models the {\it sharp arithmetic} condition is
$L(E)<\beta(\alpha)$. It is indeed ``sharp'' because for the almost
Mathieu operator, there exist examples of eigenvalues with
$L(E) =\beta(\alpha)$ \cite{transition}, and it was proved that the
spectrum is pure point for a.e. $x$ if $L(E)>\beta(\alpha)$. The latter result was
recently established for all type I operators, a large open set of
analytic potentials \cite{gj} and is conjectured to be universal for
all analytic potentials. The sharp arithmetic condition for singular
continuous spectrum (i.e. absence of eigenvalues in the regime of
positive Lyapunov exponents) was proved, effectively, for all Lipschitz
$f,$ in \cite{AYZ} (see also \cite[footnote 3]{ten_martini}). However,
as soon as the Lipschitz modulus of continuity is relaxed, the current results
 become less effective \cite{JKocic}, for example
 depending also on the H\"older exponent for H\"older $f$.

 However, the same sharp arithmetic condition works for a.e. $x$ for
 the Maryland model \cite{JL} and some other discontinuous $f$
 \cite{Damanik}. Moreover, for all Lipschitz monotone $f$ pure
 point spectrum was proved in the authors' recent work under the
 complementary arithmetic condition $L(E)>\beta(\alpha)$, all this
 suggesting that the sharp arithmetic condition for absence of point
 spectrum of the form
 $L(E)<\beta(\alpha)$ is {\it universal} for operators \eqref{eq_h_def}.

 This universality is what we prove in this paper, for a large class
 of $f.$

For a $1$-periodic $f\colon \R\to \R$, define
\beq
\label{eq_var_definition}
\var f :=\sup\l\{\sum_{j=0}^{n-1}|f(x_{j+1})-f(x_j)|\colon 0=x_0<x_1\ldots<x_n=1\r\}.
\eeq
For $B\ge 0$, let also
\beq
\label{eq_truncation_def}
[f]_B(x):=\begin{cases}
f(x),&|f(x)|\le B;\\
B,&|f(x)|>B.
\end{cases}
\eeq
We say that a $1$-periodic function $f\colon\R\to [-\infty,+\infty)$ is of {\it of semi-bounded variation} if
\beq
\label{eq_semi_variation}
\cV(f):=\sup_{B\ge 1}\frac{1}{B}\var[f]_B<+\infty.
\eeq
In particular, all $1$-periodic functions monotone on $[0,1)$ are of semi-bounded variation.

For $\alpha\in \R\setminus\Q$, let $H(x)$ be defined by \eqref{eq_h_def}. Our main result is
\begin{theorem}\label{main}
 Suppose that $f\colon \R\to [-\infty,+\infty)$ is a $1$-periodic
 function of semi-bounded variation 
 and satisfies  \eqref{eq_finite_integral}.  Then, for almost every $x\in [0,1)$, the set
$$
\{E\in \R\colon 0<L(E)<\beta(\alpha)\}
$$
can only support singular continuous spectrum of $H(x)$.
\end{theorem}
We have two immediate corollaries
\begin{cor}
\label{th_gordon1}
Suppose that $f\colon \R\to [-\infty,+\infty)$ is $1$-periodic,
non-decreasing on $[0,1)$, and satisfying
\eqref{eq_finite_integral}. Then the conclusion of Theorem
$\ref{main}$ holds.
\end{cor}
\begin{cor}
\label{th_gordon2}
Suppose that $f\colon \R\to \R$ is $1$-periodic with $\var(f)<+\infty$. Then, the conclusion of Theorem $\ref{main}$ holds.
\end{cor}

Corollary \ref{th_gordon1} is the second part of the arithmetic phase
transition theorem for quasiperiodic operators with $\gamma$-monotone
potentials, in which case it establishes also the singular continuous
spectrum for a.e. $x,$ since a large class of discontinuous $f,$ we have 
$L(E)>0$ for on a set of full spectral measure \cite{DK,Simon_Spencer}. The pure point part of the transition
theorem has been recently obtained in \cite{JK2}. The fact that $L(E)=\beta(\alpha)$ is the exact transition point (at least, for a full measure set of phases) justifies the word ``sharp'' in the title.
\subsection{Acknowledgements} SJ’s work was supported by NSF DMS--2052899, DMS--2155211, and Simons 896624.
IK's work was supported by NSF DMS--1846114, DMS--2052519, and the 2022 Sloan Research Fellowship.

\section{Some preliminaries: discrepancy and sufficient conditions for the Gordon's argument}
\subsection{Discrepancy and Koksma's inequaltiy}Let $x_1,\ldots,x_n\in [0,1]$ be real numbers. Following \cite{Kuipers}*{Definition 1.2}, define their {\it discrepancy} by
\beq
\label{eq_discrepancy1}
D_n^{\ast}(x_1,\ldots,x_n)=\sup\limits_{0<t\le 1}\l|\frac{\#\{k\colon x_{k}\in [0,t)\}}{n}-t\r|.
\eeq
Koksma's inequality (see, for example, \cite{Kuipers}*{Theorem 2.5.1}) states that, for any function $f$ on $[0,1]$ of bounded variation, we have
\beq
\label{eq_koksma}
\l|\int_0^1 f(x)\,dx-\frac{1}{n}\sum_{k=1}^n f(x_k)\r|\le \mathrm \var (f) D_n^{\ast}(x_1,\ldots,x_n).
\eeq
Let 
$$
\alpha=[a_0,a_1,a_2,\ldots]=a_0+\frac{1}{a_1+\frac{1}{a_2+\frac{1}{\ldots}}}
$$
be an irrational number represented by a continued fraction expansion. As usual, denote the continued fraction approximants to $\alpha$ by truncating the above continued fraction:
$$
\frac{p_k}{q_k}:=[a_0,a_1,a_2,\ldots,a_k]=a_0+\frac{1}{a_1+\frac{1}{a_2+\frac{1}{\ldots+\frac{1}{a_k}}}}.
$$
The following well known estimate is contained in \cite[Section 2.3]{Kuipers}, see also \cite{Khinchine,Raven}.
\begin{prop}
\label{prop_discrepancy_rotation}
In the above notation, for irrational $\alpha$ we have
\beq
\label{eq_discrepancy_q}
D_{q_k}(\{x\},\{x+\alpha\},\ldots,\{x+(q_k-1)\alpha\})\le \frac{2}{q_{k}}.
\eeq
As a consequence, 
\beq
\label{eq_discrepancy_n}
D_{n}(\{x\},\{x+\alpha\},\ldots,\{x+(n-1)\alpha\})=o(1),\quad n\to +\infty.
\eeq
\end{prop}
The rate of decay in \eqref{eq_discrepancy_n} highly depends on $\alpha$ (but can be chosen uniformly in $x$). It can be obtained by finding a sequence $q_k\to +\infty$ such that $n=s_k q_k+r_k$ with $r_k=o(n)$ (for example, making a choice in a way that $s_k\to +\infty$).

\subsection{An abstract Gordon's theorem} A sharp Gordon's theorem (that is, the one that involves the condition $\beta(\alpha)>L(E)$ rather than, say, $\beta(\alpha)>CL(E)$) was first presented in \cite{AYZ} to treat the singular continuous part of the arithmetic spectral transition for the almost Mathieu operator (see also \cite[footnote 3]{ten_martini}).  In this subsection, we will state Proposition \ref{prop_gordon}, which is an abstract versions of such theorem. For the convenience of the reader, we also include the proof, even though it can be easily extracted from the cited results. The proof of Theorem \ref{main}, therefore, reduces to verifying the assumptions of Proposition \ref{prop_gordon}.

Let $V\colon \Z\to \C$, and $q_n\to +\infty$. We will say that $V$ has {\it $\beta$-repetitions along $\{q_n\}$} if, for $n$ large enough,
\begin{equation}\label{eq_beta_repetitions}
    \max_{1\le s < q_n}|V(s)-V(s+q_n)|\le e^{-\beta q_n},\quad  \max_{1\le s < q_n}|V(s)-V(s- q_n)|\le e^{-\beta q_n}.
\end{equation}
For a Schr\"odinger operator on $\ell^2(\Z)$ with the potential $V$
 \begin{equation}\label{Schro}
    (H \psi)(n)= \psi(n+1)+\psi(n-1) + V(n)\psi(n),
\end{equation}
define one-step transfer matrices
$$
A_n(E):=\begin{pmatrix}
E-V(n)&-1\\
1&0
\end{pmatrix}
$$
and the multi-step transfer matrices by
$$
M_{n,k}(E):=\begin{cases}
A_{n-1}(E)A_{n-2}(E)\cdots A_k(E),&k\le n;\\
M_{k,n}^{-1}(E),&k>n.
\end{cases}
$$
Note that we have
$$
M_{n,k}(E)=M_{n,\ell}(E)M_{\ell,k}(E),\quad \text{for} \quad k,\ell,n\in \Z.
$$
Define also
$$
M_n(E):=M_{n,0}(E).
$$
A sequence $q_n\to +\infty$ will be called {\it $\lambda$-telescopic for $V$ at $E$} if the following bounds are satisfied for large enough $n$:
\beq
\label{eq_telescopic_regularity_1}
\|M_{2q_n,q_n+s+1}\|\|M_{s,0}\|\le e^{\lambda q_n},\quad \text{for}\quad s=0,\ldots,q_n-1;
\eeq
\beq
\label{eq_telescopic_regularity_2}
\|M_{-q_n,-q_n+s-1}\|\|M_{s,q_n}\|\le e^{\lambda q_n},\quad \text{for}\quad s=0,\ldots,q_n-1.
\eeq
\begin{prop}
\label{prop_gordon}
Suppose that $V$ has $\beta$-repetitions along a $\lambda$-telescopic sequence at $E$ with $\beta>\lambda>0$. Then $E$ is not an eigenvalue of $H$.
\end{prop}
\begin{proof}
The argument in the proof is by now standard
(see \cite{AYZ,JY,JKocic}). While most of the proofs involve additional assumptions, the actual arguments can be stated in terms of $\beta$-repetitions and $\lambda$-telescopic sequences defined above. For the convenience of the reader, we include the complete argument. Since $E$ is fixed, we will drop it from the notation. Let also $P:=\begin{pmatrix}
-1&0\\0&0	
\end{pmatrix}$. We have
$$
M_{-q_n}-M_{q_n}^{-1}=M_{-q_n,0}-M_{0,q_n}=\sum_{s=0}^{q_n-1}\l(V(s-q_n)-V(s)\r) M_{-q_n,-q_n+s-1}PM_{s,q_n},
$$
$$
M_{2 q_n}-M_{q_n}^2=(M_{2q_n,q_n}-M_{q_n})M_{q_n}=\sum_{s=0}^{q_n-1}\l(V(s+q_n)-V(s)\r) M_{2q_n,q_n+s+1}PM_{s,0}M_{q_n}.
$$
Using \eqref{eq_beta_repetitions}, \eqref{eq_telescopic_regularity_1}, and \eqref{eq_telescopic_regularity_2}, we arrive to, for every $v\in \C^2$:
\beq
\label{eq_gordon_assumptions}
\|M_{-q_n}-M_{q_n}^{-1}\|\le q_n e^{-(\beta-\lambda)q_n};\quad \|(M_{2 q_n}-M_{q_n}^2)v\|\le q_n e^{-(\beta-\lambda)q_n}\|M_{q_n}v\|.
\eeq
Suppose that $\psi\in \ell^2(\Z)$ is a non-trivial solution to $Hu=Eu$. Let $v:=\begin{pmatrix}
\psi(0)\\ \psi(-1)
\end{pmatrix}$, and assume that $\|v\|=1$. Since $|\psi(n)|\to 0$ as $n\to +\infty$, we have
\beq
\label{eq_norms}
\|M_{q_n}v\|,\,\|M_{-q_n}v\|,\,\|M_{2q_n}v\|<1/2
\eeq
for $n$ large enough. The characteristic equation
$$
M_{q_n}-(\tr M_{q_n})I+M_{q_n}^{-1}=0,
$$
applied to $v$, implies
\beq
\label{eq_trace}
|\tr M_{q_n}|\le \|M_{q_n}v\|+\|M_{-q_n}v\|+\|M_{-q_n}-M_{q_n}^{-1}\|<1
\eeq
for $n$ large enough, in view of \eqref{eq_gordon_assumptions} and \eqref{eq_norms}. However, another characteristic equation
$$
M_{q_n}^2-(\tr M_{q_n})M_{q_n}+I=0,
$$
also applied to $v$, implies
$$
\|M_{2q_n}v\|\ge \|M_{q_n}^2v\|-\|(M_{2 q_n}-M_{q_n}^2)v\|\ge 1-|\tr M_{q_n}|\|M_{q_n}v\|-\|(M_{2 q_n}-M_{q_n}^2)v\|>1/2,
$$
again, for $n$ large enough in view of \eqref{eq_gordon_assumptions}, \eqref{eq_trace}, \eqref{eq_norms}, which contradicts the last estimate in \eqref{eq_norms}.
\end{proof}



\section{A uniform upper bound for cocycles of bounded variation}
Let $M\colon \R\to \mathrm{M}_N(\mathbb C)$ be a $1$-periodic matrix-valued measurable function. A quasiperiodic cocycle (over an irrational rotation) is a pair $(\alpha,M)$, where $\alpha\in \R\setminus\Q$. The $n$-th iteration of the cocycle $(\alpha, M)$ is defined by
$$
(\alpha,M)^n=(n\alpha,M_n),\quad \text{where}\quad M_n(x)=\prod_{k=n-1}^0 M(x+k\alpha),
$$
and the (maximal) Lyapunov exponent is defined by
\begin{equation}\label{l}
L(\alpha,M):=\lim_{n\to\infty}\frac{1}{n}\int_{[0,1)}\log\|M_{n}(x)\|\,dx=
\inf_{n\in \N}\frac{1}{n}\int_{[0,1)}\log\|M_n(x)\|\,dx.
\end{equation}

We will say that $M$ is of bounded variation if all its matrix elements $M_{ij}=M_{ij}(x)$ satisfy $\var M_{ij}<+\infty$, see \eqref{eq_var_definition}. Our goal is to establish the following uniform upper bound on the iterates of linear cocycles of bounded variation.
\begin{theorem}
\label{th_uniform_upper}
Suppose that $M\colon \R\to \mathrm{M}_N(\mathbb C)$ is $1$-periodic and of bounded variation. Then, for every $\ep>0$, there exists $n_0=n_0(\ep,\alpha,M)>0$ such that for every $n\ge n_0$ one has
$$
\frac{1}{n}\log\|M_n(x)\|\le L(\alpha,M)+\ep,\quad \forall x\in [0,1).
$$
\end{theorem}
\begin{remark} In the continuous case, the result follows in a
  standard way by subadditivity, unique ergodicity of the irrational
  rotation, and compactness. In the case of bounded almost continuous cocycles\footnote{that is, continuous outside a set whose closure has measure
  zero}, the bound was obtained in \cite{LanaMavi}. Here we show that bounded variation is sufficient.
\end{remark}
We will need several preliminaries in order to proceed with the proof. The sequence $\log\|M_n(x)\|$ forms a sub-additive cocycle with $\frac{1}{n}\int_0^1\log\|M_n(x)\|\,dx$ converging to $L(\alpha,M)$. We will consider iterations of this cocycle over pieces of the trajectory of length $m$ as a subadditive cocycle over the rotation $x\mapsto x+m\alpha$. For large $m$, the interaction between these iterations behaves like  an additive cocycle, for which uniform upper bounds follow from applying Koksma's inequality and monotonicity under truncations. It will be convenient to state the following straightforward lemma separately.
\begin{lemma}
\label{lemma_scalar_upperbound}
Let 
$$
h\in \mathrm{L}^1[0,1);\quad -\infty\le h(x)\le B,\quad\forall x\in [0,1).
$$
Let also 
$$
\{r_0,\ldots,r_{n-1}\}\subset[0,1),\quad D:=D_n^{\ast}(r_0,\ldots,r_{n-1}).
$$ 
Suppose that, for some $A\ge B$, we have $\var[h]_A<+\infty$. Then
\beq
\label{eq_scalar_upperbound}
\frac{1}{n}\sum_{j=0}^{n-1}h(r_j)\le\frac{1}{n}\sum_{j=0}^{n-1}[h]_A(r_j)\le \int_0^1[h]_A(x)\,dx+D\var [h]_A.\eeq
\end{lemma}
\begin{proof}
The first inequality in \eqref{eq_scalar_upperbound} follow from the definition of $h_A$ and the fact that $A\ge B$. The second inequality follows from Koksma's inequality \eqref{eq_koksma}.
\end{proof}

{\noindent\it Proof of Theorem $\ref{th_uniform_upper}$}. 
Let $g_m(x):=\frac{1}{m}\log\|M_m(x)\|$. From subadditivity of the original cocycle, we have (in the first inequality, without loss of generality)
$$
+\infty>\int_0^1 g_m(x)\,dx\ge \int_0^1 g_{2m}(x)\,dx\ge L(\alpha,M),\quad \forall m\in \N.
$$
Since $M$ is bounded, we have
\beq
\label{eq_upper_bv_bounded}
g_m(x)\le C_1=C_1(M),\quad\forall x\in[0,1),\quad \forall m\in \N.
\eeq
and, since $\|M_m(\cdot)\|$ is of bounded variation and $[\log|\cdot|]_A$ is a Lipschitz function, we also have
\beq
\label{eq_c3}
\var[g_{m_0}]_A\le C_2=C_2(m_0,\alpha,M,A).
\eeq
We will prove the upper bound by an $\ep/4$-argument. Fix $\ep>0$ and find some $m_0>0$ and $A=A(m_0)>C_1$ such that
\beq
\label{eq_epsilon_1}
\int_0^1 g_{m_0}(x)\,dx-L(\alpha,M)<\ep/4,\quad \int_0^1 [g_{m_0}]_A(x)\,dx-\int_0^1 g_{m_0}(x)\,dx\le \ep/4.
\eeq
Note that the left hand sides of both inequalities are non-negative, and all terms are finite. Use Proposition \ref{prop_discrepancy_rotation} to find $k_0=k_0(A,m_0,\alpha,M)$ large enough so that
\beq
\label{eq_epsilon_2}
D_k:=D_{k}^{\ast}(x,x+m_0\alpha,x+2m_0\alpha,\ldots,x+(k-1)m_0\alpha)<C_2^{-1}\ep/4,\quad \forall k\ge k_0, \quad \forall x\in [0,1).
\eeq
Let $n=k m_0$ with $k\ge k_0$. Then, by expressing $M_n(x)$ as a product of $M_k(\cdot)$, one can apply Lemma \ref{lemma_scalar_upperbound} combined with \eqref{eq_epsilon_1}, \eqref{eq_epsilon_2} and obtain
\begin{multline}
\label{eq_upper_bound_subsequence}
\frac{1}{n}\log\|M_n(x)\|\le \frac{1}{k}\sum_{j=0}^{k-1}g_{m_0}(x+km_0\alpha)\le \int_0^1[g_{m_0}]_A(x)\,dx+C_2 D_k\\ \le L(\alpha,M)+3\ep/4.	
\end{multline}
This completes the proof of the theorem for $n$ divisible by $m_0$. In order to consider the general case, let $n=k m_0+m$ large enough so that $k\ge k_0$ and $0\le m<m_0$. Then
$$
\log\|M_n(x)\|\le \log\|M_{k m_0}(x)\|+m g_{m}(x+k m_0\alpha)\le k m_0(L(\alpha,M)+3\ep/4)+m_0 C_1.
$$
After dividing by $n$ and using the fact that $n\ge k m_0$, one obtains
$$
\frac{1}{n}\log\|M_n(x)\|\le L(\alpha,M)+3\ep/4+\frac{m_0 C_1}{n}\le L(\alpha,M)+3\ep/4+\frac{C_1}{k}.
$$
The proof can now be completed by choosing $k>4 C_1/\ep$.\,\,\qed
\begin{remark}
\label{rem_weakly_bounded}
In Theorem \ref{th_uniform_upper}, we only used the properties of the scalar sub-additive cocycle $g_m$. It is easy to restate the result completely in these terms. Let $\{g_m\}_{m\in \Z}$ be a sub-additive real scalar cocycle over a circle rotation. Suppose that $g_0(x)$ is bounded from above and $\var [g_m]_A<+\infty$ for all $m\in \Z$, $A\in \R$. Suppose also that $L(g)>-\infty$. Then the same conclusion holds:  for every $\ep>0$, there exists $n_0=n_0(\ep,\alpha,g)>0$ such that for every $n\ge n_0$ one has
$$
\frac{1}{n} g_n(x)\le L(\alpha,g)+\ep,\quad \forall x\in [0,1).
$$
Since we only used the fact that $D_k\to 0$, the arguments apply to all ergodic circle maps, rather than just rotations.
\end{remark}

\section{Estimates for scalar functions of semi-bounded variation}
In this section, we obtain some estimates for scalar functions of bounded variation, which will be used in obtaining $\beta$-repetitions and the $\lambda$-telescopic property later for Schr\"odinger cocycles. 
\subsection{Finite difference estimates}
The original proofs of Gordon's theorem for quasiperiodic operators require some kind of uniform differentiability (e.~g. Lipschitz property) of $f$ in order produce $\beta$-repetitions. The following lemma, as well as Corollaries \ref{cor_differentiabilty_bounded_variation} and \ref{cor_differentiabilty_semibounded_variation}, are motivated by the fact that every function of bounded variation is almost everywhere differentiable.
\begin{lemma}
\label{lemma_differentiability1}
Let $f$ be a $1$-periodic function on $\R$, non-decreasing on $[0,1)$, with 
$$
|f(x)|\le B,\quad \forall x\in\R.
$$
For every $\delta\in (0,1)$ and $A>0$, we have
\beq
\label{eq_original_set}
\l|\{x\in [0,1)\colon |f(x+\delta)-f(x)|>A\delta\}\r|\le 2\delta+6BA^{-1}.
\eeq
\end{lemma}
\begin{proof}
The set 
$$
S:=\{x\in (\delta,1-\delta)\colon f(x)<f(x)+A\delta\}
$$
satisfies the trivial inclusion
$$
S\subset \cup_{x\in S}[x,x+\delta]\subset [0,1).
$$
By Vitali's covering lemma, there exist finitely many $x_j\in S$ such that the intervals $[x_j,x_j+\delta]$ are disjoint, and the intervals $[x_j-\delta,x_j+2\delta]$ cover $S$. Let
$$
S':=\cup_j [x_j,x_j+\delta].
$$
We have, using the fact that the intervals $[x_j,x_j+\delta]$ are disjoint and that $f$ is non-decreasing:
$$
|S|\le 3 |S'|=3\sum_j \l|[x_j,x_j+\delta]\r|<3 A^{-1}\sum_j(f(x_j+\delta)-f(x_j))\le 6 A^{-1}B.
$$
Since the remaining part of the set in \eqref{eq_original_set} is contained in $[0,\delta]\cup[1-\delta,1)$, this completes the proof.
\end{proof}
\begin{cor}
\label{cor_differentiabilty_bounded_variation}
Let $f$ be a $1$-periodic function on $\R$ with $\var f<+\infty$. Then a similar conclusion holds:
For every $\delta\in (0,1)$ and $A>0$, we have
$$
\l|\{x\in [0,1)\colon |f(x+\delta)-f(x)|>A\delta\}\r|\le 4\delta+6A^{-1}\var f.
$$
\end{cor}
\begin{proof}
The function $f$ is a difference of functions of the kind considered in Lemma \ref{lemma_differentiability1} whose absolute values are bounded by $\frac12 \var f$.
\end{proof}
\begin{cor}
\label{cor_differentiabilty_semibounded_variation}
Let $f\colon \R\to [-\infty,+\infty)$ be a $1$-periodic function satisfying $\eqref{eq_finite_integral}$: that is,
$$
\int_0^1\log(1+|f(x)|)\,dx<+\infty.
$$
For every $\delta\in (0,1)$ and $A,B>0$, we have
\begin{multline}
\label{eq_original_set_2}
\l|\{x\in [0,1)\colon |f(x+\delta)-f(x)|>A\delta\}\r|\le \\ \le 4\delta+6 A^{-1}e^{B}\var[f]_{e^B}+\frac{2}{B}\int_{\{x\in [0,1)\colon\log(1+|f(x)|)>B\}}\log(1+|f(x)|)\,dx.
\end{multline}
\end{cor}
\begin{proof}
If $\var [f]_{e^B}=+\infty$, there is nothing to prove. Otherwise, apply Corollary \ref{cor_differentiabilty_bounded_variation} to $[f]_{e^B}$ and estimate the measure of the set where $f(x)\neq [f]_{e^B}(x)$ or $f(x+\delta)\neq [f]_{e^B}(x+\delta)$ by Markov's inequality.
\end{proof}

\subsection{Multiplicative upper estimates}
In this subsection, we discuss some applications of Koksma's inequality \eqref{eq_koksma} which involve estimates of sums of functions of bounded variation over finite sets with known bounds on discrepancies. They are referred to as ``multiplicative'' since later they will be applied to the function $\log(1+|f|)$. We start from the following restatement of Lemma \ref{lemma_scalar_upperbound}. Recall, as defined in \eqref{eq_truncation_def}:
$$
[h]_B(x)=\begin{cases}
h(x),&|h(x)|\le B;\\
B,&|h(x)|>B.
\end{cases}
$$
\begin{lemma}
\label{lemma_scalar_restatement}
Let 
$$
\{r_0,\ldots,r_{n-1}\}\subset[0,1),\quad D:=D_n^{\ast}(r_0,\ldots,r_{n-1}).
$$  Suppose that $h\in \rL^1[0,1)$ and, for some $B\ge 0$, we have $V=\var[h]_B<+\infty$. Then
\beq
\label{eq_scalar_upperbound_2}
\frac{1}{n}\sum_{j=0}^{n-1}[h]_B(x+r_j)\le\int_0^1 [h]_B(x)\,dx+2DV,\quad \forall x\in [0,1).
\eeq
As a consequence, 
\begin{multline}
\label{eq_consequence_1}
\l\{x\in [0,1)\colon \frac{1}{n}\sum_{j=0}^{n-1}h(x+r_j)>\int_0^1 [h]_B(x)\,dx+2DV\r\}\subset \\ \subset
\bigcup_{j=0}^{n-1}\l\{x\in [0,1)\colon h(x+r_j)>B\r\}.	
\end{multline}
As another consequence, the measure of the above set can be estimated by
$$
\l|\l\{x\in [0,1)\colon \frac{1}{n}\sum_{j=0}^{n-1}h(x+r_j)>\int_0^1 [h]_B(x)\,dx+2DV\r\}\r|\le \frac{n}{B}\int_{x\in [0,1)\colon h(x)>B}h(x)\,dx.
$$
\end{lemma}
\begin{proof}
Follows from applying Lemma \ref{lemma_scalar_upperbound} to $[h]_{B}$, in which case the exceptional set would be the set where $h(x+r_j)>B$ for some $j$.

The factor 2 can be omitted if one replaces the definition of discrepancy by a translationally invariant one.
\end{proof}

As mentioned above, Lemma \ref{lemma_scalar_restatement} will be applied to the function 
$$
h(x)=\log F(x)=\log(1+|f(x)|),
$$
where $f$ is of semi-bounded variation. Recall that the latter means
$$
\cV(f):=\sup_{B\ge 1}\frac{1}{B}\var[f]_B<+\infty.
$$ 
In order to obtain a meaningful statement, we will need some estimates on $[\log (1+|f|)]_B$. For $-\infty\le B_1<B_2\le +\infty$, define
$$
[f]_{B_1,B_2}(x)=\begin{cases}
B_1,&f(x)\le B_1;\\
f(x),&B_1<f(x)<B_2;\\
B_2,&f(x)\ge B_2.
\end{cases}
$$
Clearly, $[f]_B=[f]_{-B,B}$ for $B\ge 0$. The following lemma is elementary and we omit the proof.
\begin{lemma}
\label{lemma_additive_variation}
Suppose that $f\colon \R\to\R$ is $1$-periodic, and $B_1<B_2<B_3$. Then
$$
\var[f]_{B_1,B_3}=\var[f]_{B_1,B_2}+\var[f]_{B_2,B_3}.
$$
\end{lemma}
The next lemma shows that if $f$ is of semi-bounded variation, then so is $\log(1+|f|)$.
\begin{lemma}
\label{lemma_log_variation}
Let $F\ge 1$, $N\in \N$. Then
\beq
\label{eq_lemma_log_1}
\var[\log F]_N\le \sum_{j=0}^{\lceil N/\log 2 \rceil}2^{-j}\var[F]_{2^j,2^{j+1}}.
\eeq
As a consequence,
$$
\var[\log F]_B\le 2\log(2) B \cV(F),\quad \forall B\ge 1.
$$
\end{lemma}
\begin{proof}
The first inequality follows from Lemma \ref{lemma_additive_variation} and the fact that $\log(\cdot)$ has derivative bounded by $2^{-j}$ on $[2^j,2^{j+1}]$.  The second inequality follows from the estimate 
$$
\var[F]_{2^j,2^{j+1}}\le \var[F]_{2^{j+1}}\le 2^{j+1}\cV(F).\,\,\qedhere
$$
\end{proof}
Finally, the following lemma will be useful in obtaining factorizations of Schr\"odinger cocycles.
\begin{lemma}
\label{lemma_ratio_variation}
Suppose that $f$ is of semi-bounded variation. Then $\frac{f}{1+|f|}$ has bounded variation.
\end{lemma}
\begin{proof}
By considering positive and negative parts of $f$ and in view of Lemma \ref{lemma_additive_variation}, one can assume without loss of generality that $f\ge 0$. We have
$$
\var\frac{f}{1+f}=\var\frac{1}{1+f}=\sum_{j=0}^{+\infty}\var\l[\frac{1}{1+f}\r]_{2^{-(j+1)},2^{-j}}.
$$
Since
$$
\l(\frac{1}{1+t}\r)'=-\l(\frac{1}{1+t}\r)^2,
$$
we have
$$
\var\l[\frac{1}{1+f}\r]_{2^{-(j+1)},2^{-j}}\le 2^{-2j}\var[f]_{2^j-1,2^{j+1}-1}\le 2^{-j}\cV(f),
$$
which completes the proof after summation.
\end{proof}



\section{Factorized cocycles and proof of Theorem \ref{main}}
In this section, we will establish telescopic bounds for a class of cocycles that admit a factorization that allows to combine Theorem \ref{th_uniform_upper} with estimates of the kind obtained in Lemma \ref{lemma_scalar_restatement}, in order to obtain telescopic bounds. Afterwards, we will show that Schr\"odinger cocycles admit such factorization and complete the proof of Theorem \ref{main}. 
\subsection{Telescopic bounds for factorized cocycles}
Let $M\colon \R\to \mathrm{M}_N(\C)$ be a measurable $1$-periodic function. We will assume that the following factorization takes place:
$$
M(x)=F(x)G(x),
$$
where $G\colon \R\to \mathrm{M}_N(\C)$ is a $1$-periodic matrix-valued function of bounded variation such that $L(\alpha,G)>-\infty$, and $F(x)\ge 1$ is a scalar function, satisfying 
\beq
\label{eq_finite_integral_2}
\int_0^1\log F(x)\,dx<+\infty
\eeq
and the variation estimate
\beq
\label{eq_variation_truncated}
\var[\log F]_B<+\infty,\quad \forall B\ge 1.
\eeq
Fix some finite subset of $[0,1)$ with a uniform discrepancy bound
$$
\{r_0,\ldots,r_{n-1}\}\subset[0,1),\quad D_n^*(x+r_0,x+r_1,\ldots,x+r_{n-1})\le D,\quad \forall x\in [0,1).
$$
Define
\beq
\label{eq_ps_def}
P_s(x):=\l\|\prod_{j=n-1}^{s+1}M(x+r_j)\r\|\l\|\prod_{j=s}^{0}M(x+r_j)\r\|.
\eeq
\begin{lemma}
\label{lemma_telescopic}
Let $M(x)$, $F(x)$, $G(x)$, and $\{r_0,\ldots,r_{n-1}\}$ be as
above. Define $P_s(x)$ by $\eqref{eq_ps_def}$. Let $n$ be large enough so that the conclusion of Theorem $\ref{th_uniform_upper}$ holds for $(\alpha,G)$. Suppose that $2D\var[\log F]_B<\ep$. Then
$$
\l\{x\in(0,1)\colon P_s(x)>e^{n (L(\alpha,M)+2\ep)}\r\}\subset \l\{x\in [0,1)\colon \frac{1}{n}\sum_{j=0}^{n-1} \log F(x+r_j)>L(\alpha,F)+\ep\r\}.
$$
The set in the right hand side is, in turn, contained in
$$
\bigcup_{j=0}^{n-1}\l\{x\in [0,1)\colon F(x+r_j)> e^B\r\}.
$$
\end{lemma}
\begin{proof}
Follows from the direct applications of Theorem \ref{th_uniform_upper} for $(\alpha,M)$ and Lemma \ref{lemma_scalar_restatement} for $(\alpha,F)$ (the scalar cocycle).
\end{proof}
We will now describe a more concrete application relevant to Theorem \ref{main}. Let $\frac{p_k}{q_k}$ be a continued fraction approximant to $\alpha$, and suppose that $\delta<\frac{1}{10q_k}$. For $s=0,1,\ldots,q_k-1$, define
$$
R_s:=\{x,x+\alpha,\ldots,x+(s-1)\alpha,x+(s+1)\alpha+\delta,x+(s+2)\alpha+\delta,\ldots,x+(q_k-1)\alpha+\delta\}.
$$
The set $R_s$ has $q_k-1$ elements. It is easy to see that have
$$
D_{q_k-1}^*(R_s)\le \frac{3}{q_k}.
$$
\begin{cor}
\label{cor_telescopic}
Under the assumptions and using the notation of Lemma $\ref{lemma_telescopic}$, let $n=q_k-2$, $\{r_0,\ldots,r_{q_k-2}\}=R_s$ defined above. Then
\begin{multline*}
\bigcup_{s=0}^{q_k-1}\{x\in(0,1)\colon P_s(x)>e^{n L(\alpha,M)+2\ep}\}\subset\\ \subset \l(\bigcup_{j=0}^{q_k-1}\l\{x\in [0,1)\colon F(x+j\alpha)>e^B\r\}\r)\cup \l(\bigcup_{j=0}^{q_k-1}\l\{x\in [0,1)\colon F(x+j\alpha+\delta)>e^B\r\}\r).
\end{multline*}
As a consequence,
$$
\l|\bigcup_{s=0}^{q_k-1}\{x\in(0,1)\colon P_s(x)>e^{n L(\alpha,M)+2\ep}\}\r| \le \frac{2 q_k}{B}\int_{\{x\in [0,1)\colon F(x)>e^B\}}\log F(x)\,dx.	
$$
\end{cor}
\begin{proof}
Follows from applying Lemma \ref{lemma_telescopic} to all choices of $R_s$ and combining all possible cases.
\end{proof}
\subsection{Schr\"odinger cocycles and the proof of Theorem \ref{main}}
 We are now ready to apply the general results to Schr\"odinger cocycles. Let
\beq
\label{eq_schrodinger_factorization}
S^{f,E}(x):=\begin{pmatrix}E-f(x)&-1\\1&0 \end{pmatrix}=:F(x)G(x,E),\quad \text{where}\quad F(x):=1+|f(x)|.
\eeq

\begin{cor}
\label{cor_telescopic_schrodinger}
Define $F,G$ as in \eqref{eq_schrodinger_factorization}, and that $f$ satisfies the assumptions of Theorem $\ref{main}$: that is,
$$
\cV(f)=\sup_{B\ge 1}\frac{1}{B}\var[f]_B<+\infty,\quad \log F\in \Ell^1[0,1).
$$
Let $\ep>0$. Assume that $q_{n_j}$ is a sequence of denominators of the continued fraction expansion of $\alpha$ such that
$$
\sum_{j=1}^{+\infty}\l(\int_{\l\{x\in [0,1)\colon F(x)>e^{\ep q_{n_j}/10}\r\}}\log F(x)\,dx\r)<+\infty.
$$
Then there exists a full measure subset $X_{\mathrm{tel}}\in [0,1)$ such that, for every $x\in X_{\mathrm{tel}}$ and every $E\in \R$, the sequence $\{q_{n_j}\}$ is $(L(\alpha,E)+\ep)$-telescopic for $V(n)=f(x+n\alpha)$.
\end{cor}
\begin{proof}
Since $|\cdot|$ is Lipschitz, we have that $F(x)=1+|f(x)|$ also satisfies $\cV(F)<+\infty$. From Lemma \ref{lemma_log_variation}, the same holds for $\log F$. From Lemma \ref{lemma_ratio_variation}, the matrix elements of $G$ are of bounded variation, and therefore satisfies the assumptions of Theorem \ref{th_uniform_upper}. As a consequence,  one can apply Lemma \ref{lemma_telescopic} and Corollary \ref{cor_telescopic} with, say, $B=\frac{\ep q_k}{10}$, combined with the Borel--Cantelli lemma.
\end{proof}
\begin{remark}
\label{rem_uniform_factorizations}
An important ingredient of the proof is the fact that $F(x)$ in the factorization \eqref{eq_schrodinger_factorization} does not depend on $E$, which allows to avoid considering an intersection of uncountably many full measure subsets. While $G(x,E)$ will be of bounded variation for all $E$, we do not require any control of the dependence of $n_0(\ep,\alpha,M)$ with $M=G(\cdot,E)$ in Theorem \ref{th_uniform_upper} on $E$.
\end{remark}
It remains to establish $\beta$-repetitions. We will use Corollary \ref{cor_differentiabilty_semibounded_variation} in the following result.
\begin{cor}
\label{cor_beta_repetitions}
Suppose that $f\colon \R\to [-\infty,+\infty)$ with
$$
\cV(f)<+\infty,\quad \log F\in \Ell^1[0,1).
$$
Let $\ep>0$. Assume that $q_{n_j}$ is a sequence of denominators of the continued fraction expansion of $\alpha$ such that
$$
\sum_{j=1}^{+\infty}\l(\int_{\l\{x\in [0,1)\colon F(x)>e^{\ep q_{n_j}/10}\r\}}\log F(x)\,dx\r)<+\infty,\quad q_{n_j+1}\ge e^{\beta q_{n_j}}.
$$
Then there exists a full measure subset $X_{\mathrm{rep}}\in [0,1)$ such that, for every $x\in X$ and every $E\in \R$, the potential $V(n)=f(x+n\alpha)$ has $(\beta-\ep)$-repetitions along $q_{n_j}$.
\end{cor}
\begin{proof}
Apply \eqref{eq_original_set_2} from Corollary \ref{cor_differentiabilty_semibounded_variation} with $B=\frac{\ep q_{n_j}}{10}$, $\delta=\pm\dist(q_{n_j}\alpha,\Z)$ and, say, $A=e^{\ep q_{n_j}/5}$.
\end{proof}
The proof of Theorem \ref{main} is now obtained directly by combining Proposition \ref{prop_gordon} with Corollaries \ref{cor_telescopic_schrodinger} and \ref{cor_beta_repetitions}.\,\,\qed

\end{document}